\documentclass[a4paper]{article}
\usepackage{ifthen}

\usepackage{amsthm,amsmath,amsfonts,graphicx}
\usepackage[T1]{fontenc}
\usepackage{amsmath}
\usepackage{amssymb}
\usepackage{hyperref}

\def\N{\mathbb{N}}
\def\R{\mathbb{R}}

\DeclareMathOperator{\conv}{conv}

\newtheorem{theorem}{Theorem}

\newtheorem{definition}{Definition}
\newtheorem{proposition}{Proposition}

\newcommand{\defgl}{\mathop{:=}}
\renewcommand{\rho}{\varrho}

\title{On the Geometry of Holmsen's Combinatorial Version of the Colorful Carath\'eodory}
\author{Helena Bergold and  Winfried Hochst\"attler \\ FernUniversit\"at in Hagen, Germany}
\date{}

\begin{document}
%
\maketitle
\setlength{\parindent}{0em}
\begin{abstract}
\noindent	Carath\'eodorys Theorem of convex hulls plays an important role in convex geometry. In 1982, B\'ar\'any formulated and proved a more general version, called the Colorful Carath\'eodory. This colorful version was even more generalized by Holmsen in 2016. He formulated a combinatorial extension in \cite{Holmsen} and found a topological proof. Taking a dual point of view we gain an equivalent formulation of Holmsen's result that has a more geometric meaning. 
\end{abstract}

\section{Carath\'eodory's Theorem}

The following theorem, named after Constantin Carath\'eodory, plays an important role in convex geometry since it gives an upper bound to the length of the convex combination of a point lying in a convex set. 
\begin{theorem}[Carath\'{e}odory]
	For $P \subseteq \R^d$ ($d \in \N$) and a point $x \in \conv P$ there are $d+1$ points $p_0, \ldots, p_d \in P$ such that  $ x \in \conv\{p_0, \ldots, p_d\}$.
\end{theorem}

   B\'ar\'any proved the Colorful Carath\'eodory \cite{Barany} in 1982, a more general version of Carath\'eodory's Theorem. In this version, the set $P$ is divided into different subsets. If these sets are disjoint, we can interpret them as colors, such that each point gets a color dependent on the subset in which the point is included. 
  
\begin{theorem}[Colorful Carath\'{e}odory, B\'ar\'any (1982) \cite{Barany}] \label{THM_ColofulCaratheodory}
	Consider $d+1$ sets $P_0, \ldots , P_d$ in $\R^d$ ($d \in \N$). If $x \in \R^d$ is in all convex hulls $\conv P_i$ ($0 \leq i \leq d$), there are points $p_i \in P_i$ in every $P_i$ ($ 0 \leq i \leq d$) such that $x$ is in the convex hull of those $d+1$ points, i.e. $ x \in \conv\{p_0, \ldots , p_d \}$.
\end{theorem}
 This theorem tells us that the point $x$ is in a colorful simplex, which means that each corner is in a different color, i.e. belongs to a different set $P_i$. 
 

It is even possible to extend the Colorful Carath\'eodory. In \cite{Holmsen} Holmsen considered a combinatorial and a topological extension of Theorem \ref{THM_ColofulCaratheodory} which are equivalent to each other. 

\section{Introduction to Matroid and Oriented Matroid Theory}
To understand the combinatorial version, we will shortly introduce the necessary terms in the theory of matroids and oriented matroids. For detailed information we refer to \cite{Oxley}, \cite{Sturmfels} and \cite{BachemKern}.
\begin{definition} \label{DEF_matroid}
	A \emph{matroid} $\mathcal{M}$ on a \emph{ground set} $E$ with a collection $\mathcal{C}$ of circuits is an ordered pair $\mathcal{M} = (E, \mathcal{C})$ such that the following conditions hold: 
	\begin{enumerate}
		\item[(M1)] $\emptyset \notin \mathcal{C}$.
		\item[(M2)] For $X,Y \in \mathcal{C}$ with $ X \subseteq Y$ is $ X = Y$. 
		\item[(M3)] If $ X,Y \in \mathcal{C}$, $X \neq Y$ and $ e \in X \cap Y$, there is a $ Z  \in \mathcal{C}$ such that $Z~\subseteq~(X~\cup~Y) \backslash \{e\}$.
	\end{enumerate}
	The elements of $\mathcal{C}$ are called \emph{circuits}. 
	A \emph{loop} is an element $ x\in E$ such that~$\{x\}$ is a circuit. 
	 An \emph{independent set} of a matroid $ \mathcal{M}=(E, \mathcal{C})$ is a subset $X \subseteq E$ such that no subset of $X$ is a circuit. A maximal independent set is called a \emph{basis}. 
	
\end{definition} 




\begin{definition} \label{DEF_Rankfunction}
	The \emph{rank} of a set $ X \subseteq E$ is the cardinality of the inclusion-maximal independent subset of $X$. We denote the rank of $X$ as $\rho(X)$. The \emph{rank of a matroid} $\mathcal{M}$ is $\rho(E)$.
\end{definition}
 

\begin{definition}
	A subset $D \subseteq E$ of a matroid $\mathcal{M}$ is a \emph{double circuit} if  
	\begin{align*}
		\rho(D) &= |D| -2 \qquad \text{and} \\
		\forall e \in D: \: \rho(D \backslash \{e\}) &= |D| -2.
	\end{align*}
\end{definition}

\begin{definition}
	The \emph{dual matroid} $\mathcal{M}^\ast$ of a matroid $\mathcal{M} = (E, \mathcal{C})$ is a matroid on the same ground set $E$ such that a set is independent if and only if the complement contains a basis of $\mathcal{M}$. It is $\mathcal{M}^{\ast \ast} = \mathcal{M}$.
	The circuits of the dual matroids are called \emph{cocircuits} of $\mathcal{M}$. 	
\end{definition}

To define the notion of oriented matroids, we need to introduce signed sets. A \emph{signed subset} of a finite set $E$ is an ordered pair $X = (X^+,X^-)$ with $X^+,X^-~\subseteq~E$ such that 
$X^+ \cap X^- = \emptyset$.
We write $-X = (X^-,X^+)$ and $\underline{X} = X^+ \cup X^- \subseteq E$. 

\begin{definition} \label{DEF_OrientedMatroid}
	A pair $\mathcal{O}=(E, \mathcal{C})$ is an oriented matroid if 
	\begin{enumerate}
		\item[(O1)] $(\emptyset, \emptyset) \notin \mathcal{C}$ and if $X \in \mathcal{C}$, then $-X \in \mathcal{C}$.
		\item[(O2)] For two sets $X,Y \in \mathcal{C}$ with $ \underline{X} \subseteq \underline{Y}$, it is either $X = Y$ or $X = -Y$. 
		\item[(O3)] For $X,Y \in \mathcal{C}$ with $X \neq -Y$ and $ e \in X^+ \cap Y^-$ there is a $Z \in \mathcal{C}$ such that $ Z ^+ \subseteq (X^+ \cup Y^+) \backslash \{v\}$ and $ Z ^- \subseteq (X^- \cup Y^+) \backslash \{v\}$.
	\end{enumerate}
	
\end{definition}
The signed sets in $\mathcal{C}$ are called \emph{(signed) circuits} of $\mathcal{O}$ and a \emph{positive circuit} is a (signed) circuit $(X^+, \emptyset)$ of $\mathcal{O}$. If $\underline{X} \subseteq S$ for a set $S$ and a signed set $X$, we say~$X$ \emph{is contained} in $S$. \newline

It is easy to see that the set $\underline{\mathcal{C}} = \{ \underline{X} : \: X \in \mathcal{C} \}$ for a circuit system $\mathcal{C}$ of an oriented matroid is the collection of circuits of a matroid. This matroid is called the \emph{underlying matroid} of $\mathcal{O}$. The \emph{rank of an oriented matroid} is defined as the rank of its underlying matroid.

\begin{definition}
	Two signed sets $X,Y$ are \emph{orthogonal} if $\underline{X} \cap \underline{Y} = \emptyset$ or the restrictions of $X$ and $Y$ to their intersection $\underline{X} \cap \underline{Y}$ is neither equal nor opposite.
\end{definition}

For any oriented matroid $\mathcal{O} =(E,\mathcal{C})$, there is a unique maximal family $\mathcal{C}^\ast$ such that $X$ and~$Y$ are orthogonal for all $X \in \mathcal{C}$ and $Y \in \mathcal{C}^\ast$. This set is a set of circuits of an oriented matroid on the ground set $E$, called the \emph{dual oriented matroid} of~$\mathcal{O}$.
The circuits of the dual matroid are called \emph{cocircuits}.	\par
The matroid underlying the dual oriented matroid is the dual of the underlying matroid.

\section{Holmsen's Theorem}

The following theorem is the theorem of Holmsen stated in \cite{Holmsen}. 
In this section we will only show the connection of this version to the Colorful Carath\'eodory (Theorem \ref{THM_ColofulCaratheodory}). 
\begin{theorem}[Holmsen] \label{THM_Holmsen}
	 Let $\mathcal{O}$ be an oriented matroid on the ground set $E$ with rank $r$. Consider a matroid $\mathcal{M}$ on the same ground set $E$ and with rank function $\rho$ such that $\rho(E) > r$. If every subset $S \subseteq E$ with $\rho(E \backslash S ) < r$ contains a positive circuit of $\mathcal{O}$, there exists a positive circuit of $\mathcal{O}$ which is independent in $\mathcal{M}$.
\end{theorem}

At first glance you may not notice the connection of this theorem to the Colorful Carath\'eodory Theorem. For this reason, we will recall from \cite{Holmsen} that the Colorful Carath\'eodory is a special case. 
\begin{proof}[Holmsen's Theorem $\Rightarrow$ Colorful Carath\'eodory]
	Let $P_0, P_1, \ldots, P_d$ be subsets of $\R^d$ and $x \in \R^d$ a point such that $x \in \conv P_i$ for all $i \in \{1, \ldots, d \}$.
	Consider the disjoint union $E = P_0 \dot{\cup} P_1 \dot{\cup} \cdots \dot{\cup} P_d$ as a multiset of points in $\R^d$. We assume $ x \notin E$ (if $x$ is already in one of the sets $P_i$ ($i \in \{1,\ldots,d\}$), the statement is trivial).
	For every inclusion-minimal subset $S \subseteq E$ with a linear dependency $\sum_{p \in S} \alpha_p(p-x) =0$ ($\alpha_p \in \R$), we consider the two subsets of $S$: 
	\begin{align*}
		S^+ = \{ p \in S: \: \alpha_p >0 \},\\
		S^- = \{ p \in S: \: \alpha_p <0\}.
	\end{align*}
	All signed subsets $(S^+,S^-)$ constructed in such a way build an oriented matroid $\mathcal{O}=(E, \mathcal{C})$ where $\mathcal{C}$ is the set of signed subsets $(S^+,S^-)$.
	This so constructed oriented matroid has the rank $r$ equal to the dimension of the affine span of $E \cup \{x\}$.
	 \par 
	Observe that for any set $S \subseteq E$ and any $y \in \conv S$ there is linear dependency with only positive coefficients. This shows that 
	\begin{align*}
		x \in \conv S \quad \Leftrightarrow \quad S \text{ contains a positive circuit of } \mathcal{O}.
	\end{align*}
	 By the preliminaries of the Colorful Carath\'eodory $x$ is contained in all $\conv P_i$. By this observation we know that all $P_i$ contain a positive circuit. \par 
	Furthermore, we define the matroid $\mathcal{M}$ by its independent sets. So let $I \subseteq E$ be independent if and only if  $|I \cap P_i| \leq 1$ for all $i \in \{1, \ldots , d\}$. This is a so called partition matroid and hence it fulfills the matroid axioms. 
	We need to check whether every subset $S \subseteq E$ with $\rho(E \backslash S) < r $ contains a positive circuit in order to apply Theorem \ref{THM_Holmsen}. The condition $\rho(E \backslash S) <r$ means there are at most $r-1$ elements in $E \backslash S$, so this set has an empty intersection with $d+1-(r-1) = d-r+2 \geq 2$ of the $d+1$ sets $P_0, \ldots , P_{d}$. This shows that $S$ contains two of the sets $P_0, \ldots , P_{d}$. Each $P_i$ ($i =0, \ldots, d$) contains a positive circuit and so does $S$ itself. \par 
	Now all necessary conditions of Theorem \ref{THM_Holmsen} are fulfilled, so there is a positive circuit $C$ of $\mathcal{O}$ such that $ C \in \mathcal{I}$. By the observation above this means $ x~\in~\conv C$. In $C$ there is at most one point of every set $P_i$ so we can extend $C$ to a set consisting of exactly one point of every $P_i$ as mentioned in the Colorful Carath\'eodory Theorem. 
\end{proof}

\section{Dual of Holmsen's Theorem}

Using the basics introduced previously we are finally able to reformulate Holmsen's Theorem.  \\

%
\textbf{Claim 1:} \textit{We can replace the condition ``for all $S \subseteq E$ with $\rho(E \backslash S )<r$'' by ``for all $S \subseteq E$ with $\rho(E\backslash S)=r-1$'' in Theorem \ref{THM_Holmsen}.} 
\begin{proof}
	Every subset $S$ with $\rho(E \backslash S )=r-1$ fulfills obviously the condition that the rank of the complement is smaller than $r$. If we start with a subset $S \subseteq E$ such that $\rho(E \backslash S ) < r-1$, we extend the set $E \backslash S$ to $ E \backslash S'$ such that $\rho(E \backslash S' )=r-1$ (possible since the rank of $\mathcal{M}$ is greater than $r$). This implies $S' \subseteq S$. The smaller set $S'$ contains by assumption an positive circuit which is also contained in $S$.
\end{proof}

\textbf{Claim 2:}
\textit{It is enough to show Theorem \ref{THM_Holmsen} for matroids with rank $r+1$. In particular, matroids of rank r+2 or higher need not to be considered
explicitely.}
\begin{proof}
	To see this we assume that $\mathcal{O}$ is an oriented matroid of rank $r$ and $\mathcal{M}$ a matroid of rank $> r+1$, both on the same ground set $E$. We delete elements of $E$ such that, we get a submatroid $\mathcal{M}'$ of $\mathcal{M}$ of rank $r+1$ on the ground set $E' \subset E$. Let $r-k$ for $k \in \N_0$ be the rank of the oriented matroid $\mathcal{O}'$, the restriction of~$\mathcal{O}$ to $E'$. We add $k$ new elements $e_1, \ldots , e_k$ as a loop to~$\mathcal{M}'$ and in such a way that they are coloops in the underlying matroid of $\mathcal{O}'$. The new matroid is denoted by $\mathcal{M}' + \{e_1, \ldots , e_k\}$ and the oriented matroid by $\mathcal{O}' + \{e_1, \ldots , e_k\}$. The rank of the new matroid $\mathcal{M}' + \{e_1, \ldots , e_k\}$ with elements $E' \cup \{e_1, \ldots , e_k\}$ is $r+1$ since the added elements do not appear in any bases. The rank of the new oriented matroid increases by $k$. We need to show that the condition ``for all $S \subseteq E' \cup \{e_1, \ldots , e_k\}$ with $\rho((E' \cup \{e_1, \ldots, e_k\}) \backslash S)~=~r~-~1$, $S$ contains a positive circuit of $\mathcal{O}' + \{e_1, \ldots , e_k\}$'' holds in the constructed matroid and oriented matroid. This is true since for any such subset $S$ with $\{e_1, \ldots , e_k\}~\subseteq~S$, we get a positive circuit $C$ of $\mathcal{O}' + \{e_1, \ldots , e_k \}$ contained in $S$. Since a coloop is in no circuit, we get $C \cap \{e_1, \ldots , e_k\} = \emptyset$. For any other subset $S$ with $\rho((E'~\cup~\{e_1, \ldots, e_k\}) \backslash S)~=~r-1$, we can extend the set $S$ by the elements $e_1, \ldots, e_k$ and get a positive circuit without $\{e_1, \ldots e_k\}$, so the positive circuit is contained in $S$ itself. So the precondition is true. Hence there is a positive circuit $C'$ of $\mathcal{O}' + \{e_1, \ldots , e_k\}$ (with $C' \cap \{e_1, \ldots , e_k\} = \emptyset $) which is independent in $\mathcal{M}' + \{e_1, \ldots , e_k\}$. This implies that $C'$ is independent in $\mathcal{M}$ as well.
	This shows that the condition ``$\rho(E)>r$'' can be replaced by ``$\rho(E) = r+1$''.
\end{proof}

So we showed that Theorem \ref{THM_Holmsen} is an equivalent formulation of
\begin{theorem} \label{THM_SimplifiedHolmsen}
	Let $\mathcal{O}$ be an oriented matroid on the ground set $E$ with rank $r$. Consider a matroid $\mathcal{M}$ on the same ground set $E$ and with rank function $\rho$ such that $\rho(E)=r+1$. If every subset $S \subseteq E$ with $\rho(E \backslash S ) = r-1$ contains a positive circuit of $\mathcal{O}$, there exists a positive circuit of $\mathcal{O}$ which is independent in $\mathcal{M}$.
\end{theorem}

By the Topological Representation Theorem every oriented matroid can be realized as an oriented pseudosphere arrangement. 
A circuit in such an arrangement is a collection of half-spaces with empty intersection. 
Since cocircuits correspond to vertices in the arrangement and so they are easier to visualize, we will present a dual version of Holmsen's Theorem.

\begin{theorem}[Dual version of Holmsen's Theorem] \label{THM_Dual}
	Let $\mathcal{M}$ be a matroid with rank function $\rho$ such that $\rho(E) = r-1$. Furthermore we consider an oriented matroid~$\mathcal{O}$ with rank $r$ on the same ground set $E$. If every double circuit of $\mathcal{M}$ contains a positive cocircuit of $\mathcal{O}$, then there exists a positive cocircuit in $\mathcal{O}$, whose complement spans $\mathcal{M}$.
\end{theorem}

\begin{proof}
	Let $\mathcal{M}, \mathcal{O}$ be as in the theorem. 
	From the rank formula 
	\begin{align*} \label{PROP_Rankformula}
		\rho_{\mathcal{M}^\ast}(X) = |X| + \rho_{\mathcal{M}}(E \backslash X) - \rho_{\mathcal{M}}(E),
	\end{align*}
	we see that $r^\ast \defgl |E|-r$ is the rank of the dual oriented matroid $\mathcal{O}^*$ and
	\begin{align*}
		\rho_{\mathcal{M}^\ast}(E) &= |E|+ \rho_{\mathcal{M}}(E\backslash E)- \rho_{\mathcal{M}}(E) \\
		&= |E| + \rho_{\mathcal{M}}(\emptyset) -(r-1) \quad = r^\ast +1.
	\end{align*}
	So $\mathcal{M}^\ast, \mathcal{O}^\ast$ fulfill the assumptions of Theorem \ref{THM_SimplifiedHolmsen}.
	Furthermore every inclusion-minimal subset $S \subseteq E$ with $\rho_{\mathcal{M}^\ast}(E \backslash S) = r^\ast - 1$ is a double circuit in its dual, which follows from the rank formula and of $\mathcal{M}^{\ast\ast} = \mathcal{M}$:
	\begin{align*}
		\rho_{\mathcal{M}}(S) = |S| + \rho_{\mathcal{M}^\ast}(E \backslash S) - \rho_{\mathcal{M}^\ast}(E) = |S| + r^\ast - 1 - (r^\ast+1) = |S| -2.
	\end{align*}
	Since $S$ is inclusion-minimal with this property, the complement of $S \backslash \{e\}$ for every $e \in S$ has  rank $\rho_{\mathcal{M}^\ast}(E \backslash (S \backslash \{e\})) = r^\ast$. This implies that 
	\begin{align*}
		\rho_{\mathcal{M}}(S\backslash \{e\}) = |S| -1 + r^\ast -(r^\ast +1) =  |S| -2
	\end{align*}
	for every $e \in E$. So $S$ is a double circuit in $\mathcal{M}$. Every set $S' \subseteq E$ which is not inclusion-minimal contains an inclusion-minimal set $S$ and hence it contains a double circuit. Every double circuit contains a positive cocircuit of $\mathcal{O}$ by assumption and so every $S$ contains a positive circuit of $\mathcal{O}^\ast$. So by Theorem~\ref{THM_SimplifiedHolmsen}, we know that there is a positive circuit of $\mathcal{O}^\ast$ which is independent in $\mathcal{M}^\ast$. This is a positive cocircuit of $\mathcal{O}$ whose complement spans~$\mathcal{M}$.
\end{proof}

The dual version mentioned in Theorem \ref{THM_Dual} is an equivalent formulation of Holmsen's Theorem. The proof that Holmsen's Theorem follows from the dual version mentioned above works analogously. \par 

\section{Complementary Positive Cocircuits}
We will now think of an oriented matroid of rank $d+1$ as an arrangement of signed pseudospheres in $S^d$. The positive cocircuits are the vertices of the main polytope. The complement of a cocircuit corresponds exactly to the spheres going through a vertex of the polytope. There are at least $d$ spheres intersecting in a vertex $v$. 
So if the complement of a positive cocircuit is considered to span the matroid~$\mathcal{M}$ of rank $d$, this means that the elements corresponding to the spheres, intersecting in the cocircuit, span the matroid.
If the complement of a positive cocircuit $C$ does not span the matroid, the rank is less than $\rho(E) = d$, so 
\begin{align*}
	\rho(E \backslash C ) \leq d-1 \leq |E \backslash C| -1.
\end{align*}
This shows that $E \backslash C$ contains a circuit of $\mathcal{M}$. \par 
If the vertex $v$ of the main polytope is degenerate and $E \backslash C$ is non-spanning or $\rho(E \backslash C) \leq d-2$ there exists a vertex $V'$ of the main polytope such that the set of spheres intersecting in $v'$ is disjoint from those intersecting in $v$. We call such a vertex \emph{complementary vertex}. \par 
We will now study in which cases the union of the elements intersecting in two adjacent (in the 1-skeleton of the face lattice of the positive polytope) positive cocircuits contain a double circuit of the matroid, i.e. in which case the edge connecting those vertices has a complementary vertex. For this reason note that:
 
\begin{proposition}
	A set $X \subseteq E$ contains a double circuit if and only if the rank of $X$ is $\rho(X) \leq |X| -2$.
\end{proposition}  

\begin{proposition} \label{Prop_ContainsDoubleCircuit}
	Every non-spanning set in a matroid of rank $d$ with more than~$d$ elements contains a double circuit.
\end{proposition} 


So by the last Proposition \ref{Prop_ContainsDoubleCircuit}, we may assume that in each vertex there are exactly $d$ intersecting elements. \par 
Let $C_1$ and $C_2$ be two adjacent positive cocircuit 
and each complement does not span the matroid. So
\begin{align*}
	\rho(E \backslash C_i) \leq d-1 \text{ for } i=1,2.
\end{align*}
If $(E \backslash C_1) \cap (E \backslash C_2)$ is independent, the submodularity of the rank function shows
\begin{align*}
	\rho\left((E \backslash C_1) \cup (E \backslash C_2)\right) &\leq \rho(E \backslash C_1) + \rho(E \backslash C_2) - \rho\left((E \backslash C_1) \cap (E \backslash C_2)\right) \\
	& \leq (d-1) + (d-1) - \left|(E \backslash C_1) \cap (E \backslash C_2)\right| \\
	& \leq d-1.
\end{align*}
Furthermore $(E \backslash C_1) \cup (E \backslash C_2)$ contains 
$d+1$ elements, so  
$(E \backslash C_1) \cup (E \backslash C_2)$ contains a double circuit if $(E \backslash C_1) \cap (E \backslash C_2)$ is independent. 

In a similar way, we get that $(E \backslash C_1) \cup (E \backslash C_2)$ contains a double circuit if either the additional element of $E\backslash C_1$ 
or the additional element of $ E\backslash C_2$ does not increase the rank.\par 
 By the assumption of the dual version (Theorem \ref{THM_Dual}) every double circuit contains a positive cocircuit. So the elements intersecting in the two cocircuits $C_1$ and~$C_2$ contain a positive cocircuit. \par 
The remaining part is that both elements increase the rank. In this case the rank of the union is
\begin{align*}
	\rho(E\backslash C_1 \cup E \backslash C_2) &= \rho(E\backslash C_1 \cap E \backslash C_2) +2 \\
	&< \: \: |E\backslash C_1 \cap E \backslash C_2| \;  +2 \\
	&= \: \: |E\backslash C_1 \cup E \backslash C_2|.
\end{align*}

In this case, the union does not contain a double circuit if the intersection does not contain one. On the other hand  $E\backslash C_1 \cup E \backslash C_2$ has rank $d$ and still contains a spanning positive cocircuit if $C_1,C_2$ do not span the matroid.

\bibliographystyle{abbrv}
\bibliography{bib}
\end{document}